\documentclass[
final
]{dmtcs-episciences}

\usepackage[utf8]{inputenc}
\usepackage{subfigure}
\usepackage{amsmath}
\usepackage{amssymb}
\usepackage{amsfonts}
\usepackage{amsthm}
\usepackage{thm-restate}
\usepackage{hyperref}
\usepackage{cleveref}
\usepackage{xcolor}
\usepackage{mathtools}
\usepackage{macros}
\usepackage{nicefrac}

\usepackage{tikz}
\usetikzlibrary{shapes.geometric,hobby,calc} 
\usetikzlibrary{decorations.pathmorphing}
\usetikzlibrary{decorations.text}
\usetikzlibrary{shapes.misc}
\usetikzlibrary{decorations,shapes,snakes}

\usepackage[round]{natbib}

\author{Meike Hatzel\affiliationmark{1}
  \and Marcin Pilipczuk\affiliationmark{2}
  \and Pawe\l{} Komosa\affiliationmark{2}
  \and Manuel Sorge\affiliationmark{2}}
\title[Constant Congestion Brambles]{Constant Congestion Brambles\thanks{\parbox[t]{0.99\textwidth}{\parbox[t]{0.85\textwidth}{The research leading to the results presented in this paper was partially carried out during the Parameterized Algorithms Retreat of the University of Warsaw, PARUW~2020, held in Krynica-Zdrój in February 2020.
This research is part of projects that have received funding from the European Research Council (ERC) under the European Union's Horizon 2020 research and innovation programme
Grant Agreements 648527 DISTRUCT (Meike Hatzel) and 714704 CUTACOMBS (Marcin Pilipczuk, Manuel Sorge).}\quad\parbox[t]{0.05\textwidth}{~\\[-3mm]\includegraphics[width=30px]{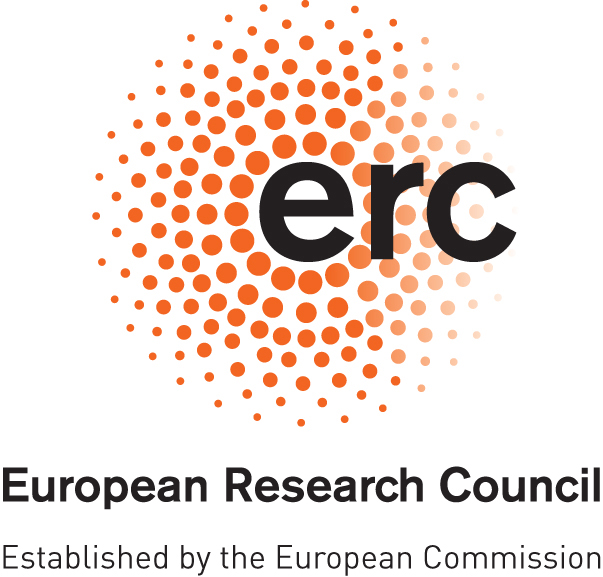}\\\includegraphics[width=30px]{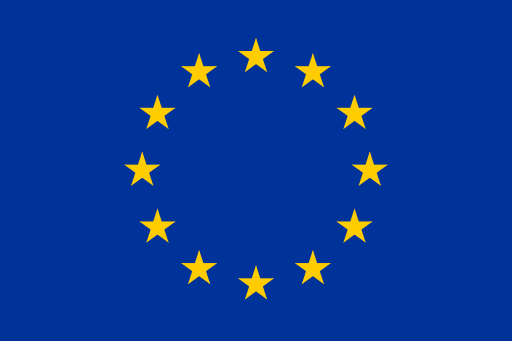}}}}}
\affiliation{
  TU Berlin, Germany\\
  University of Warsaw, Poland}
\received{2020-08-06}

\revised{2021-10-12}

\accepted{2022-02-09}

\begin{document}
\publicationdetails{24}{2022}{1}{6}{6699}
\maketitle
\begin{abstract}
	A \emph{bramble} in an undirected graph $G$ is a family of connected
	subgraphs of $G$ such that for every two subgraphs $H_1$ and $H_2$ in the bramble
	either $V(H_1) \cap V(H_2) \neq \emptyset$ or there is an edge of $G$ with one endpoint
	in $V(H_1)$ and the second endpoint in~$V(H_2)$. 
	The \emph{order} of the bramble is the minimum size of a vertex set that intersects
	all elements of a bramble.
	
	Brambles are objects dual to treewidth: As shown by Seymour and Thomas,	the maximum order of a bramble in an undirected graph $G$ equals one plus the treewidth of $G$.
	However, as shown by Grohe and Marx, brambles of high order may necessarily be of exponential size:
	In a constant-degree $n$-vertex expander a bramble of order $\Omega(n^{1/2+\delta})$ requires size
	exponential in $\Omega(n^{2\delta})$ for any fixed $\delta \in (0,\frac{1}{2}]$.
	On the other hand, the combination of results of Grohe and Marx	and Chekuri and Chuzhoy shows that a graph of treewidth $k$
	admits a bramble of order~$\widetilde{\Omega}(k^{1/2})$ and size $\widetilde{\Oh}(k^{3/2})$.
	($\widetilde{\Omega}$ and $\widetilde{\Oh}$ hide polylogarithmic divisors and factors, respectively.)
	
	In this note, we first sharpen the second bound by proving that 
	every graph $G$ of treewidth at least~$k$ 
	contains a bramble of order
	$\widetilde{\Omega}(k^{1/2})$ and congestion $2$, i.e., 
	every vertex of $G$ is contained in at most two elements of the bramble (thus the bramble is of size linear in its order). 
	Second, we provide a tight upper bound for the lower bound of Grohe and Marx:
	For every $\delta \in (0,\frac{1}{2}]$, every graph $G$ of treewidth at least $k$
	contains a bramble of order $\widetilde{\Omega}(k^{1/2+\delta})$
	and size~$2^{\widetilde{\Oh}(k^{2\delta})}$. 
\end{abstract}

\section{Introduction}
\emph{Treewidth} is a well-known measure of how tree-like a graph is:
For example, a graph is a forest if and only if it has treewidth at most $1$, while an $n$-vertex
clique has treewidth $n-1$.
The notion of treewidth, coined by the Graph Minors project~\citep{RobertsonS84},
has found many applications both in graph theory and algorithm design.
(The definition of tree decompositions and treewidth can be found in \cref{sec:prelims}.)

The notion of a \emph{bramble} is an elegant and tight obstacle to treewidth.
Given an undirected graph $G$, a bramble $\mathcal{B}$ is a family of connected subgraphs of $G$
such that for every $H_1,H_2 \in \mathcal{B}$, either $V(H_1) \cap V(H_2) \neq \emptyset$
or there is an edge of $G$ with one endpoint in $V(H_1)$ and one endpoint in~$V(H_2)$. 
The measure of complexity of a bramble is its order: a \emph{hitting set} of a bramble $\mathcal{B}$
is a set $X \subseteq V(G)$ such that $X \cap V(H) \neq \emptyset$ for every $H \in \mathcal{B}$,
and the \emph{order} of a bramble is the minimum size of such a hitting set. 
A simple argument shows that for every bramble $\mathcal{B}$ in $G$, every tree decomposition of $G$
needs to contain a bag hitting $\mathcal{B}$, thus the treewidth of $G$ plus one bounds
the maximum order of a bramble in $G$. 
The beauty of the bramble definition lies in the (highly nontrivial) fact that the above
relation is tight: There is always a bramble in $G$ of order equal to the treewidth of $G$ plus one~\citep{SeymourT93}.

However, while treewidth has found numerous applications in algorithm design, 
the use of brambles in algorithms is scarce. 
The main reason for that lies in the result of~\cite{GroheM09}:
While a bramble provides a dual object tightly related to treewidth, it can be of 
size exponential in the graph.
In particular, for every $\delta \in (0,1]$, in any $n$-vertex constant-degree expander
the treewidth is $\Omega(n)$, but any bramble of order $\Omega(n^{1/2+\delta})$
has size\footnote{\cite{GroheM09} formally only proved a $\Omega(n^{\delta})$
	lower bound, but a slightly more careful analysis of their calculations shows a
	lower bound of $\Omega(n^{2\delta})$ in the exponent.} exponential in $\Omega(n^{2\delta})$.
Hence, to certify that the treewidth is larger than $n^{1/2+\delta}$, one is required 
to look at exponential-size brambles.

On the positive side, \cite{GroheM09} proved that every $n$-vertex
graph of treewidth~$k$ admits a bramble of order $\Omega(\sqrt{k}/\log^2 k)$
and size $\Oh(k^{3/2} \log n)$. 
Combining this with the result of \cite{ChekuriC15} stating
that every graph $G$ of treewidth $k$ admits a topological minor of maximum degree
$3$, $\Oh(k^4)$ vertices, and treewidth at least $k/p(\log k)$ for some polynomial $p$, 
one obtains that a graph of treewidth $k$ admits a bramble of order $\widetilde{\Omega}(\sqrt{k})$
and size $\widetilde{\Oh}(k^{3/2})$. Here, $\widetilde{\Omega}(\cdot)$ and $\widetilde{\Oh}(\cdot)$
omit polylogarithmic divisors and factors, respectively. 

\medskip

In this note, we provide the following two strengthenings and tightenings of the results
of Grohe and~Marx. 

First, we improve the positive result with respect to the bramble size to brambles of congestion $2$.
A bramble $\mathcal{B}$ in a graph $G$ is of \emph{congestion} $c$ if every vertex of $G$
is in at most $c$ elements of $\mathcal{B}$.
Clearly, the order of a bramble~$\mathcal{B}$ of congestion $c$ is at least $|\mathcal{B}|/c$,
so, in brambles of constant congestion, the size and order are within a constant multiplicative
factor of each other. 
A bramble of congestion $1$ implies a clique minor of the same size.
Thus large grids show that it is possible to have arbitrarily
large treewidth without having a bramble of congestion $1$ and size $5$.
In contrast, our strengthening of the results of \cite{GroheM09} shows that 
there is always a bramble of order $\widetilde{\Omega}(\sqrt{k})$ and congestion $2$.
\begin{restatable}{theorem}{firstresult}
	\label{thm:sqrt}
	There exists a polynomial $p(\cdot)$ such that for every positive integer $k$
	and every graph $G$ of treewidth at least $k$ 
	the graph $G$ contains a bramble of order at least $\sqrt{k}/p(\log k)$ and congestion~$2$.
\end{restatable}

On a high level, the proof of \cref{thm:sqrt} follows similar lines as the construction
of treewidth sparsifiers by~\cite{ChekuriC15}.
A graph of the required treewidth $k$ contains a large so-called strong path-of-sets system,
as shown by~\cite{Chekuri2016}.
From such a system we can build an auxiliary graph whose vertices represent long paths which can be arbitrarily interlinked by pairwise disjoint paths.
On the vertex set of this auxiliary graph, we can play what is called the cut-matching game.
By a result from~\cite{khandekar2009graph} on this game there is a strategy to construct an expander subgraph of the auxiliary graph within $\BigO{\Brace{\log{k}}^2}$ rounds of adding perfect matchings.
We then transfer this back to the path-of-sets system by embedding each round in a different set to obtain the expander as something similar to a minor (models of vertices might intersect, but only twice).
A large enough expander contains a large clique as a minor; this was shown by~\cite{kawarabayashi2010separator}.
The minor models of the clique vertices then provide the desired bramble.

Second, we provide a tight matching bound (up to polylogarithmic factors) to 
the Grohe-Marx lower bound on the size of a bramble.
\begin{restatable}{theorem}{secondresult}
	\label{thm:lb}
	There exists a polynomial $p(\cdot)$ such that 
	for every constant $\delta \in (0, 0.5]$ and every integer~$k$,
	every graph $G$ of treewidth at least $k$
	contains a bramble of order at least
	$k^{0.5+\delta}/p(\log k)$ 
	and with at most $2^{k^{2\delta} \cdot p(\log k)}$ elements.
\end{restatable}
Here, the construction follows the general ideas of the construction of~\cite{GroheM09} of the bramble of order $\Omega(\sqrt{k}/\log^2 k)$ and size $\Oh(k^{3/2} \log n)$,
but with different parameters and more elaborate probabilistic analysis. 

\medskip

After introducing notation and toolbox from previous works in \cref{sec:prelims},
we prove \cref{thm:sqrt} in \cref{sec:sqrt}
and \cref{thm:lb} in \cref{sec:lb}.

\section{Preliminaries}\label{sec:prelims}

\subsection{Notation and basic definitions}
For each $n \in \mathbb{N}$ we use $[n]$ to denote $\{1, \ldots, n\}$.
Let $G$ be a graph.
A graph $H$ is a \emph{minor} of $G$ if it can be obtained from $G$ by a series of edge deletions, vertex deletions, and edge contractions.
We can also consider a minor to be a map $f: \V{H} \to 2^{\V{G}}$ such that $\Fkt{f}{v}$ is connected for all $v \in \V{H}$, $\Fkt{f}{u} \cap \Fkt{f}{v} = \emptyset$ for $u \neq v$ and if $uv \in \E{H}$ then there are $u' \in \Fkt{f}{u}$ and $v' \in \Fkt{f}{v}$ with $u'v' \in \E{G}$.
The map $f$ is called a \emph{model} of $H$ in $G$ and we refer to $\Fkt{f}{v}$ as the \emph{model of vertex $v$} for all $v \in \V{H}$.
A \emph{subdivision} of a graph $H$ is a graph that can be obtained from $H$ by a series of edge subdivisions, that is, replacing an edge $uv$ by a new vertex $w$ and two new edges $uw$ and $wv$.
A graph $H$ is a \emph{topological minor} of $G$ if a subdivision of $H$ is isomorphic to a subgraph of~$G$.

Though we avoid working with the definition directly, we also define the treewidth of $G$.
A \emph{tree-decomposition} of $G$ is a tuple $\Brace{T,\beta}$ where $T$ is a tree and $\beta: \V{T} \to 2^{\V{G}}$ a map of the tree vertices to subsets of vertices in $G$ called \emph{bags}.
The map $\beta$ has to have the following properties:
\begin{enumerate}
	\item Every vertex $v \in \V{G}$ occurs in some bag.
	\item If $v \in \Fkt{\beta}{t} \cap \Fkt{\beta}{t'}$ for $t \neq t'$, then $v \in \Fkt{\beta}{t''}$ for all $t''$ lying on the unique path between $t$ and $t'$ in $T$.
	\item For every $uv \in \E{G}$ there is a $t\in \V{T}$ such that $u,v \in \Fkt{\beta}{t}$.
\end{enumerate}
The \emph{width} of $\Brace{T,\beta}$ is defined as $\max\Set{\Abs{\Fkt{\beta}{t}} - 1 \mid t \in \V{T}}$ and the \emph{treewidth} of $G$ is defined to be the minimum width of a tree-decomposition of $G$.

A \emph{linkage} $\mathcal{L}$ in $G$ is a set of vertex-disjoint paths.
We say it is an $A$-$B$-linkage for $A,B \subseteq \V{G}$, if all its paths start in $A$ and end in $B$, are otherwise disjoint from $A$ and $B$ and $\Abs{\mathcal{L}} = \Abs{A} = \Abs{B}$.

\begin{definition}[well-linked]
	A set of vertices $X$ in a graph $G$ is \emph{well-linked} if for all $A,B \subseteq X$ with $\Abs{A}=\Abs{B}$ there is an $A$-$B$-linkage in $G-\Brace{X\setminus \Brace{A \cup B}}$.
\end{definition}

\begin{definition}[bramble]
	A \emph{bramble} $\mathcal{B}$ in $G$ is a collection of connected subgraphs $B_1,\dots,B_s$ such that for any two elements $B_i,B_j$ we have $\V{B_i} \cap \V{B_j} \neq \emptyset$ or there exists $e=uv \in \E{G}$ with $u \in \V{B_i}$ and $v \in \V{B_j}$.
	The bramble $\mathcal{B}$ is of \emph{size} $s$.
	A \emph{hitting set} of $\mathcal{B}$ is a vertex subset $H \subseteq \V{G}$ such that for all $i \in [s]$ we have $H \cap \V{B_i} \neq \emptyset$.
	The \emph{order} of $\mathcal{B}$ is the minimum size of a hitting set of $\mathcal{B}$, i.e.\@ $\min\Set{\Abs{H}\mid H \subseteq \V{G} \text{ such that } \forall i \in [s] \colon H \cap B_i \neq \emptyset}$.
	The \emph{congestion} of bramble $\mathcal{B}$ is the maximum, taken over all vertices $v \in V(G)$, of the number of elements that contain~$v$, i.e.\@ $\max_{v \in \V{G}} \Abs{\Set{B_i \mid v \in \V{B_i}}}$.
\end{definition}

\begin{definition}[expander]
	A graph $G$ is an \emph{$\alpha$-expander} if for every partition $\Brace{S,S'}$ of the vertex set with $S,S' \neq \emptyset$ we have
	\begin{align*}
		\frac{\Abs{\E{S,S'}}}{\min\Set{\Abs{S},\Abs{S'}}} \geq \alpha,
	\end{align*}
	where $\E{S,S'}$ is the set of edges in $G$ that have one endpoint in $S$ and the other in $S'$.
\end{definition}

\subsection{Treewidth sparsifiers}

We use the following treewidth sparsification result of~\cite{ChekuriC15}.
\begin{theorem}[\cite{ChekuriC15}, Theorem 1.1]\label{thm:tw-sparsifier}
	There exists a polynomial $p(\cdot)$ such that
	for every integer~$k$, every graph~$G$ of treewidth at least $k$
	contains a topological minor with 
	(i)~$\Oh(k^4)$~vertices, (ii)~maximum degree $3$, and (iii) treewidth
	at least $k / p(\log k)$. 
\end{theorem}
\Cref{thm:tw-sparsifier} asserts that, for the cost of some polylogarithmic factors on the treewidth we can assume that the considered graph has maximum degree $3$ and size polynomial in the treewidth when deriving a bramble from a given graph.

\subsection{The cut-matching game}

The \emph{cut-matching game} is a two-player game.
The two players are called the \emph{cut} player and the \emph{matching} player.
They play in turns on an even-size set $V$ of vertices, building a (multi)graph $H$ with $V(H)=V$.
Initially the graph~$H$ has no edges.

In each turn, the cut player chooses a partition $\Brace{A,B}$ of $V$ into two equal-size sets.
Then the matching player chooses a matching $M$ between $A$ and $B$.
The matching is added to $H$ (with multiplicities, i.e., if one of the edges of $M$ is already
present in $H$, an additional copy is added, increasing the multiplicity). 
If the graph $H$ is an $\Fkt{\Omega}{1}$-expander at the end of a round the game ends.
If the game ends we say it \emph{yields} the new graph $H$.
We can consider the graph $H$ as \emph{consisting} of the matchings~$M_1,\dots,M_r$ chosen throughout the $r$ rounds of the game.

\begin{theorem}[\cite{khandekar2009graph}, Lemma 3.1 and 3.2]\label{thm:cut-matching}
	There is a strategy for the cut player that, with high probability, yields an $\Fkt{\Omega}{1}$-expander after $\BigO{\Brace{\log h}^2}$ rounds, where $h = |V|$.
\end{theorem}

A typical application for the cut-matching game is as follows. 
Let $G$ be a graph and let $X \subseteq V(G)$ be a well-linked set of even size $h = |X|$.
Consider a cut-matching game played on $V=X$.
The matching player is simulated by a flow computation:
For every partition of $X$
into $A$ and~$B$, the graph $G$ contains an $A$-$B$-linkage $\mathcal{P}$ by well-linkedness of~$X$.
The returned matching of the matching player corresponds to how the paths of $\mathcal{P}$ match the vertices
of $X$. 
Then, if the cut player plays the strategy of~\cite{khandekar2009graph}, after $\Oh((\log h)^2)$
rounds it obtains an expander; note that this expander has maximum degree $\Oh((\log h)^2)$
and can be embedded (with congestion) into $G$ as a union of $\Oh((\log h)^2)$ linkages.

The concept of path-of-sets systems introduced in the following subsection can be used to ensure that the linkages are disjoint and thus, to the price of constant congestion elsewhere, allow us to ged rid of the congestion caused when embedding the expander into $G$.

\subsection{Path-of-sets systems}

The following definition was introduced and used by~\cite{Chekuri2016} to prove a polynomial bound for the excluded-grid theorem.
We use it here in conjunction with a cut-matching game to obtain an expander graph.

\begin{definition}[path-of-sets system]
	Let $G$ be a graph.
	A \emph{path-of-sets system} of \emph{width} $h$ and \emph{length} $r$ (also called \emph{$\Brace{h,r}$-path-of-sets system}) in $G$ is a tuple $\Brace{\mathcal{S},\mathcal{A},\mathcal{B},\mathcal{P}}$ consisting of three sequences of pairwise disjoint vertex sets $\mathcal{S} = S_1, \dots, S_r$, $\mathcal{A} = A_1, \dots, A_r$ and $\mathcal{B} = B_1, \dots, B_r$, and a sequence of linkages $\mathcal{P} = \mathcal{P}_1,\dots,\mathcal{P}_{r-1}$ whose paths are pairwise disjoint such that
	\begin{enumerate}
		\item for all $i \in [r]$ the graph $\InducedSubgraph{G}{S_i}$ is connected,
		\item for all $i \in [r]$ we have $A_i \subseteq S_i$, $B_i \subseteq S_i$, $\Abs{A_i}=\Abs{B_i}=h$, and $A_i \cap B_i = \emptyset$, as well as for all $A \subseteq A_i$ and $B \subseteq B_i$ of same size there is an $A$-$B$-linkage within $\InducedSubgraph{G}{S_i}$,
		\item for all $i \in [r-1]$ we have $\mathcal{P}_i$ consists of $h$ disjoint $B_i$-$A_{i+1}$-paths that are internally disjoint to any set of $\mathcal{S}$.
	\end{enumerate}
	
	A path-of-sets system is called \emph{strong}, if for all $i \in [r]$ we have that $A_i$ is well-linked in~$G[S_i]$ and so is~$B_i$.
\end{definition}
We will indeed only use strong path-of-sets systems.

Chekuri and Chuzhoy proved that every graph with large enough treewidth contains a path-of-sets system of large length and width.

\begin{theorem}[\cite{Chekuri2016}, Theorem 3.5]\label{thm:chekuri-path-of-sets}
	There exist two constants $c_1, c_2 > 0$ and a polynomial-time randomised algorithm that, given a graph $G$ of treewidth $k$ and integers $w, \ell > 2$, such that $k/\Brace{\log k}^{c_1} > c_2 w \ell^{48}$, with high probability returns a strong path-of-sets system of width $w$ and length $\ell$ in $G$.
\end{theorem}
Note that the above implies that such a path-of-sets system always exists, because, by the statement of the theorem, there is a non-zero probability that the algorithm successfully produces such a system.
\begin{corollary}
	\label{thm:path-of-sets}
	There exists a polynomial $p(\cdot, \cdot)$ with positive coefficients and a function $\Fkt{f}{h,r} = hr^{48}p(\log h,\log r)$ such that, for all integers $h, r \geq 2$, every graph~$G$ of treewidth at least $\Fkt{f}{h,r}$ contains a strong path-of-sets system of width $h$ and length $r$.
\end{corollary}
\begin{proof}
	Let $c_1$, $c_2$ be the constants from \cref{thm:chekuri-path-of-sets} and let $c'_1 = \lceil c_1 \rceil$.
	Pick $p(x, y) = c_2((48 + 2c'_1)(\log(c_3) + x + y))^{c'_1}$, where we specify $c_3 > 1$ later.
	Denote $f = f(h, r)$ and take a graph $G$ of treewidth at least $f$.
	We claim that we may apply the algorithm of \cref{thm:chekuri-path-of-sets} to~$G$, with $w = h$ and $\ell = r$.
	It suffices to show that
	\begin{equation}
		\frac{f}{\Brace{\log f}^{c_1}} > c_2 h r^{48}\text{.}
		\label[equation]{eq:path-of-sets1}
	\end{equation}
	Since $f = hr^{48}p(\log h,\log r)$, we have $c_2 f / p(\log h, \log r) = c_2 h r^{48}$.
	Thus, 
	\begin{equation}
		p(\log h, \log r)/c_2 > (\log f)^{c_1}
		\label{eq:path-of-sets2}
	\end{equation}
	implies \cref{eq:path-of-sets1}.
	To see \cref{eq:path-of-sets2}, observe that the following from bottom to top imply \cref{eq:path-of-sets2}
	\begin{align*}
		((48 + 2 c'_1)\log(c_3 h r))^{c'_1} & > (\log f)^{c_1} \\
		(48 + 2 c'_1)\log(c_3 h r) & > \log f \\
		(c_3 h r)^{48} \cdot (c_3 h r)^{2 c'_1} & > hr^{48} c_2 ((48 + 2 c'_1)(\log(c_3) + \log h + \log r))^{c'_1} \\
		(c_3 h r)^{2c'_1} & > c_2 ((48 + 2 c'_1)(\log(c_3 h r)))^{c'_1} \\
		\frac{(c_3 h r)^{2}}{\log(c_3 h r)} & > (c_2)^{1/c'_1} (48 + 2c'_1) \\ 
		c_3 h r & > (c_2)^{1/c'_1} (48 + 2c'_1) \text{.}
	\end{align*}
	Thus, \cref{eq:path-of-sets1} holds for an appropriate choice of $c_3$, as required.
\end{proof}

This structure is used in the proof of \cref{thm:sqrt} to combine it with the cut-matching game defined above.
Every set $S_i$ is used to embed one of the matchings chosen by the matching player, this way none of them intersect in $G$.
This of course brings up the new problem of them not being on the same vertex set, which can be taken care of with the vertices being represented by paths instead of single vertices.
The whole construction then only has congestion two, as paths and linkages can meet within the $S_i$.

\subsection{Expanders contain cliques}

The reason for us to construct an expander in the given graph is that it is known that expanders do contain large clique minors.
We can use the minor model of a clique to construct a bramble, as seen later, so expanders are closely related to brambles.

\begin{theorem}[\cite{kawarabayashi2010separator}]
	\label{thm:cliques-in-expanders}
	There is a constant~$c > 0$ such that every $\alpha$-expander~$G$ on $n$ vertices and maximum degree at most $d$ contains a clique on at least $c \alpha \sqrt{n}/d$ vertices as a minor.
\end{theorem}
\begin{proof}
	This is a simple corollary of a theorem of~\cite{kawarabayashi2010separator}.
	They showed that there is a constant~$c' > 0$ and a constant $n_0 \in \mathbb{N}$ such that, for every $n, t \in \mathbb{N}$ with $n \geq n_0$, if $H$ is a graph with $n$ vertices that does not contain a clique minor on $t$ vertices, then $H$ has a $2/3$-separator of size at most $c' t \sqrt{n}$ (see \cite{kawarabayashi2010separator}, Theorem 1.2).
	Herein, a $2/3$-separator is a vertex subset $S \subseteq V(H)$ such that each connected component in $H - S$ has size at most $2|V(H)|/3$. 
	
	Now let $G$ be as in \cref{thm:cliques-in-expanders}.
	First, put $c$ small enough so that if $n < n_0$, then $c \alpha \sqrt{n}/d < 1$.
	By Kawarabayashi and Reed's theorem, it now suffices to show that we furthermore may choose the constant $c$ such that $G$ does not have a $2/3$-separator of size at most $c' t \sqrt{n}$ for $t = c \alpha \sqrt{n}/d$.
	Indeed, pick any vertex subset $S \subseteq V(G)$ of size at most $c' t \sqrt{n} = c' c \alpha n/d$.
	If $S$ is a $2/3$-separator, then we claim that we may take the union~$W$ of the vertex sets of some connected components in $G - S$ such that $n/4 \leq |W| \leq n/2$.
	To see that such a union $W$ exists, consider the following.
	If there is a component of size larger than $n/2$, then the other components give the desired~$W$: Since the largest component~$C$ has size at most $2n/3$, the remaining components contain at least $\ell$ vertices, where $\ell = n/3 - |S| \geq n/3 - c' c \alpha n/d$.
	Since $\alpha \leq d$ we have $\ell \geq n/3 - c'c n$.
	By putting $c$ small enough, we have $\ell \geq n/4$.
	Thus the union of the vertex sets of components other than $C$ give the desired~$W$ if there is a component of size larger than~$n/2$.
	Otherwise, if there is a component of size at least $n/4$, then this component gives the desired~$W$.
	Otherwise, iteratively add to the initially empty union~$W$ the smallest components in order of increasing size until their total size exceeds~$n/4$.
	Note that $|W| \leq n/2$ because there is no component of size at least $n/4$. 
	Thus, indeed, we may choose $W$ as a union of connected components in $G - S$ such that $n/4 \leq |W| \leq n/2$.
	
	Since $G$ is an $\alpha$-expander and has maximum degree at most $d$, we then have $|N(W)| \geq \alpha |W|/ d \geq \alpha n/(4 d)$.
	Since $N(W) \subseteq S$, picking any $c$ satisfying $c < 1/(4c')$ thus yields that $S$ is not a $2/3$-separator, as required.
\end{proof}

\subsection{Lov\'{a}sz Local Lemma}

In the analysis in \cref{sec:lb} we will need the Lov\'{a}sz Local Lemma.
The following simplified variant suffices.

\begin{theorem}[See \cite{AlonS04}, Lemma 5.1.1]\label{thm:lll}
	Let $\mathcal{A} = \{A_1, \ldots, A_n\}$ be a finite set of events over some probability space.
	Let $\Delta \in \mathbb{N}$ such that each event in $\mathcal{A}$ is independent of all but at most $\Delta$ other events in $\mathcal{A}$.
	Suppose that there is a real number $x$ with $0 \leq x < 1$ such that for all $i \in [n]$ we have $\mathrm{Pr}(A_i) \leq x \cdot (1 - x)^\Delta$.
	Then $ \mathrm{Pr}(\bigwedge_{i = 1}^n \overline{A_i}) \geq (1 - x)^n$.
\end{theorem}

\section{Brambles of high order and congestion two}\label{sec:sqrt}

In this section we prove the first result (\cref{thm:sqrt}) of this note, namely that every graph with large enough treewidth contains a bramble of high order and low congestion.


We recall \cref{thm:sqrt}:
\firstresult*

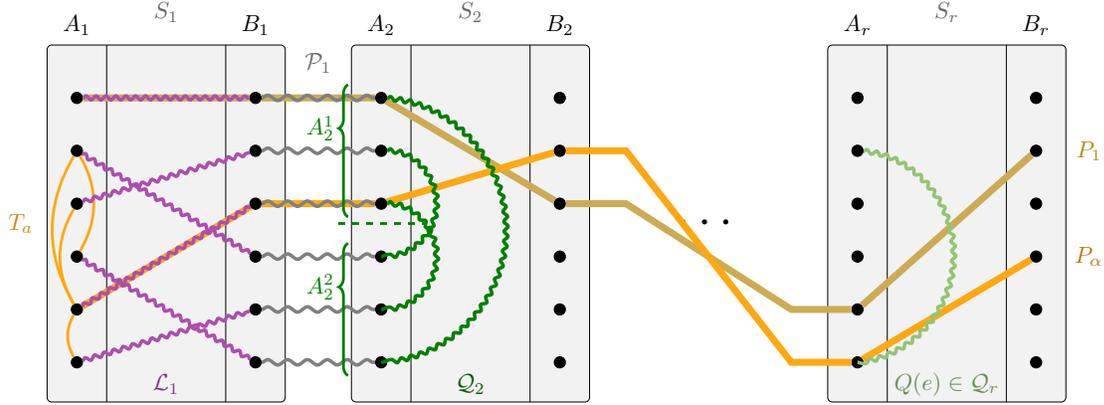
\begin{figure}[t]
	\resizebox{\textwidth}{!}{%
	\begin{tikzpicture}
		\tikzstyle{vertex}=[shape=circle, fill=black, draw, inner sep=.6mm]
		
		\pgfdeclarelayer{background}
		\pgfdeclarelayer{foreground}
		\pgfsetlayers{background,main,foreground}
		
		\def\x{1.8}
		\def\y{\x/2*3}
		
		\begin{pgfonlayer}{background}
			\node (S1) at (0,0) {};
			\node (S2) at ($(S1)+(2*\x+1,0)$) {};
			\node (Sgap) at ($(S2)+(2*\x,0)$) {\huge{$\cdots$}};
			\node (Sr) at ($(Sgap)+(2*\x,0)$) {};
			
			\node (S1label) at ($(S1)+(0,\y+0.5)$) {\textcolor{black!60}{$S_1$}};
			\node (S2label) at ($(S2)+(0,\y+0.5)$) {\textcolor{black!60}{$S_2$}};
			\node (Srlabel) at ($(Sr)+(0,\y+0.5)$) {\textcolor{black!60}{$S_r$}};
			
			\draw [fill=myLightgrey,rounded corners=2pt] ($(S1)-(\x,\y)$) rectangle ($(S1)+(\x,\y)$);
			\draw [fill=myLightgrey,rounded corners=2pt] ($(S2)-(\x,\y)$) rectangle ($(S2)+(\x,\y)$);
			\draw [fill=myLightgrey,rounded corners=2pt] ($(Sr)-(\x,\y)$) rectangle ($(Sr)+(\x,\y)$);
			
			\node (A1) at ($(S1)-(\x/2,0)$) {};
			\node (A1label) at ($(A1)+(-\x/4,\y+0.3)$) {$A_1$};
			\draw ($(A1)-(0,\y)$) -- ($(A1)-(0,-\y)$);
			
			\node (B1) at ($(S1)+(\x/2,0)$) {};
			\node (B1label) at ($(B1)+(\x/4,\y+0.3)$) {$B_1$};
			\draw ($(B1)-(0,\y)$) -- ($(B1)-(0,-\y)$);
			
			\node (A2) at ($(S2)-(\x/2,0)$) {};
			\node (A2label) at ($(A2)+(-\x/4,\y+0.3)$) {$A_2$};
			\draw ($(A2)-(0,\y)$) -- ($(A2)-(0,-\y)$);
			
			\node (B2) at ($(S2)+(\x/2,0)$) {};
			\node (B2label) at ($(B2)+(\x/4,\y+0.3)$) {$B_2$};
			\draw ($(B2)-(0,\y)$) -- ($(B2)-(0,-\y)$);
			
			\node (Ar) at ($(Sr)-(\x/2,0)$) {};
			\node (Arlabel) at ($(Ar)+(-\x/4,\y+0.3)$) {$A_r$};
			\draw ($(Ar)-(0,\y)$) -- ($(Ar)-(0,-\y)$);
			
			\node (Br) at ($(Sr)+(\x/2,0)$) {};
			\node (Brlabel) at ($(Br)+(\x/4,\y+0.3)$) {$B_r$};
			\draw ($(Br)-(0,\y)$) -- ($(Br)-(0,-\y)$);
			
			\foreach \i in {1,2,r}
			{
				\node (A-\i-vertices) at ($(A\i)+(-\x/4,\y)$) {};
				\foreach \j in {1,2,3,4,5,6}
				{
					\node[draw,vertex] (a-\j-\i) at ($(A-\i-vertices)+(0,-\j*0.8)$) {};
				}
			}
			
			\foreach \i in {1,2,r}
			{
				\node (B-\i-vertices) at ($(B\i)+(\x/4,\y)$) {};
				\foreach \j in {1,2,3,4,5,6}
				{
					\node[draw,vertex] (b-\j-\i) at ($(B-\i-vertices)+(0,-\j*0.8)$) {};
				}
			}
		\end{pgfonlayer}
		
		\begin{pgfonlayer}{foreground}
			\node (P1label) at ($(B1)!0.5!(A2)+(0,\y*0.9)$) {\textcolor{black!20!gray}{$\mathcal{P}_1$}};
			\foreach \i in {1,2,3,4,5,6}
			{
				\draw [decorate,decoration={snake,amplitude=.4mm,segment length=3mm},line width=1.5pt,gray] (b-\i-1) -- (a-\i-2);
			}
			
			\node (J1label) at ($(A1)!0.5!(B1)+(0,-\y*0.9)$) {\textcolor{black!30!myLightViolet}{$\mathcal{L}_1$}};
			\draw [decorate,decoration={snake,amplitude=.3mm,segment length=1.5mm},line width=1.5pt,myLightViolet] (a-1-1) -- (b-1-1);
			\draw [decorate,decoration={snake,amplitude=.3mm,segment length=1.5mm},line width=1.5pt,myLightViolet] (a-2-1) -- (b-4-1);
			\draw [decorate,decoration={snake,amplitude=.3mm,segment length=1.5mm},line width=1.5pt,myLightViolet] (a-3-1) -- (b-2-1);
			\draw [decorate,decoration={snake,amplitude=.3mm,segment length=1.5mm},line width=1.5pt,myLightViolet] (a-4-1) -- (b-6-1);
			\draw [decorate,decoration={snake,amplitude=.3mm,segment length=1.5mm},line width=1.5pt,myLightViolet] (a-5-1) -- (b-3-1);
			\draw [decorate,decoration={snake,amplitude=.3mm,segment length=1.5mm},line width=1.5pt,myLightViolet] (a-6-1) -- (b-5-1);

			\node (Q2label) at ($(A2)!0.5!(B2)+(0,-\y*0.9)$) {\textcolor{black!30!myGreen}{$\mathcal{Q}_2$}};
			\path[draw,dashed,very thick,myGreen] ($(S2)-(2,0)$) -- ($(S2)-(0.7,0)$); 
			\draw [decorate,decoration={brace,amplitude=4pt,aspect=0.67},very thick,myGreen] ($(a-3-2)+(-0.5,-0.2)$) -- ($(a-1-2)+(-0.5,0.2)$) node[pos=0.67, left, xshift=-2pt] {$A^1_2$};
			\draw [decorate,decoration={brace,amplitude=4pt,aspect=0.67},very thick,myGreen] ($(a-6-2)+(-0.5,-0.2)$) -- ($(a-4-2)+(-0.5,0.2)$) node[pos=0.67, left, xshift=-2pt] {$A^2_2$};
			
			\draw[decorate,decoration={snake,amplitude=.3mm,segment length=1.5mm},line width=1.5pt,myGreen] (a-1-2) to[quick curve through ={($(S2)+(0.5,0)$)}] (a-6-2);
			\draw[decorate,decoration={snake,amplitude=.3mm,segment length=1.5mm},line width=1.5pt,myGreen] (a-2-2) to[quick curve through ={($(a-3-2)+(0.8,0)$)}] (a-4-2);
			\draw[decorate,decoration={snake,amplitude=.3mm,segment length=1.5mm},line width=1.5pt,myGreen] (a-3-2) to[quick curve through ={($(a-4-2)+(0.8,0)$)}] (a-5-2);
			
			\node (Qelabel) at ($(Ar)!0.5!(Br)+(0,-\y*0.9)$) {\textcolor{black!30!pistachio}{$Q(e) \in \mathcal{Q}_r$}};
			\draw[decorate,decoration={snake,amplitude=.3mm,segment length=1.5mm},line width=1.5pt,pistachio] (a-2-r) to[quick curve through ={($(Sr)$)}] (a-6-r);
		\end{pgfonlayer}
		
		\begin{pgfonlayer}{main}
			\node (path1label) at ($(b-2-r)+(0.8,0)$) {\textcolor{darkgoldenrod}{$P_1$}};
			\path[draw,line width=3pt,white!30!darkgoldenrod] (a-1-1) -- (b-1-1) -- (a-1-2) -- (b-3-2) -- ($(b-3-2)+(1,0)$) -- ($(a-5-r)-(1,0)$) -- (a-5-r) -- (b-2-r);
			
			\draw (a-2-1) edge[bend left,very thick,myOrange] (a-4-1);
			\draw (a-2-1) edge[bend right,very thick,myOrange] (a-5-1);
			\draw (a-3-1) edge[bend right,very thick,myOrange] (a-5-1);
			\draw (a-6-1) edge[bend left,very thick,myOrange] (a-5-1);
			\node (Talabel) at ($(A1)+(-1.3,0)$) {\textcolor{black!20!myOrange}{$T_{a}$}};
			\node (pathAlabel) at ($(b-4-r)+(0.8,0)$) {\textcolor{black!30!myOrange}{$P_{\alpha}$}};
			\path[draw,line width=3pt,myOrange] (a-5-1) -- (b-3-1) -- (a-3-2) -- (b-2-2) -- ($(b-2-2)+(1,0)$) -- ($(a-6-r)-(1,0)$) -- (a-6-r) -- (b-4-r);
		\end{pgfonlayer}
		
	\end{tikzpicture}}
	\caption{$\mathcal{P}_1$ is the $B_1$-$A_2$-linkage and $\mathcal{L}_1$ the $A_1$-$B_1$-linkage that both exist by definition of the path-of-sets system.
		The path $P_1$ is given as one example of the paths starting in $A_1$ and ending in $B_r$, within $S_1$ it uses the path in $\mathcal{L}_1$ starting in the right vertex and then it continues with the path of $\mathcal{P}_1$ starting in the vertex the path of $\mathcal{L}_1$ ends in.
		In $S_2$ we see the partition $(A_2^1, A^2_2)$ of $A_2$ chosen in the second round of the cut matching game together with the linkage $\mathcal{Q}_2$ providing the answer of the matching player.
		In $S_r$ the picture shows the path $Q(e) \in \mathcal{Q}_r$ for the edge $e \in M_r$ with endpoint $\alpha$.
		The yellow edges in $A_1$ represent the spanning tree $T_a$ of the model of a vertex $a$ in the clique minor obtained from $H$.
		\label{fig:proof_thm_1}
	}
\end{figure}

\begin{proof}
	Let $G$ be of treewidth at least $k$.
	Let $q$ be the polynomial and $f$ be the function in \cref{thm:path-of-sets}.
	Let $c \geq 1$ and $h'_0$ be constants such that the cut player wins the cut-matching game within $c(\log h')^2$ rounds on a vertex set of size $h'$, $h' \geq h'_0$, when using the strategy from \cref{thm:cut-matching}.
	Pick
	\[h = \left\lfloor \frac{ k }{  q(\log k, c(\log k)^2 + 1) \cdot (c (\log k)^2 + 1)^{48} } \right\rfloor\]
	and $r = \lfloor c (\log h)^2 + 1\rfloor$.
	Note that we may add a large enough constant to the polynomial $p$ in order to make the guarantee of \cref{thm:sqrt} trivial if $h \leq h'_0$. 
	Assume thus that the cut player wins in at most $r$~rounds on a vertex set of size~$h$.
	Observe that
	\[f(h, r) = hr^{48} q(\log h, \log r) \leq \frac{ k \cdot (c (\log h)^2 + 1)^{48} \cdot q(\log h, \log r) }{  q(\log k, c(\log k)^2 + 1) \cdot (c (\log k)^2 + 1)^{48} } \leq \frac{ k \cdot q(\log h, \log r) }{ q(\log k, c(\log k)^2 + 1) } \]
	where the second inequality holds since $h \leq k$ because $c \geq 1$ and $q$ has only positive coefficients.
	Furthermore, for the same reasons we have $q(\log h, \log r) \leq q(\log k, c(\log k)^2 + 1)$ and thus $f(h, r) \leq k$.
	
	It follows that $G$ has treewidth at least $f(h, r)$ and thus, by \cref{thm:path-of-sets}, graph $G$ contains a strong $(h,r)$-path-of-sets system $\Brace{\mathcal{S},\mathcal{A},\mathcal{B},\mathcal{P}}$.
	Let $\mathcal{S}=S_1,\dots, S_r$, $\mathcal{A} = A_1,\dots, A_r$, $\mathcal{B} = B_1,\dots,B_r$, and $\mathcal{P} = \mathcal{P}_1, \ldots, \mathcal{P}_{r-1}$.
	For illustration of the remaining part of the proof, see \cref{fig:proof_thm_1}.
	By property ii) of path-of-sets systems $\Brace{\mathcal{S},\mathcal{A},\mathcal{B},\mathcal{P}}$ contains an $A_i$-$B_i$-linkage $\mathcal{L}_i$ within $\InducedSubgraph{G}{S_i}$ for all $1 \leq i \leq r$.
	Also for all $1 \leq i < r$, $\mathcal{P}_i$ is a $B_i$-$A_{i+1}$-linkage which is vertex-disjoint from each set $S_j$, $j \in [r]$ within the path-of-sets system, except for $S_i$ and $S_{i+1}$ to which $\mathcal{P}_i$ is disjoint except for the endpoints of its paths.
	We thus combine the linkages $\mathcal{L}_1,\dots,\mathcal{L}_r$ and $\mathcal{P}_1,\dots,\mathcal{P}_{r-1}$ to obtain $h$ pairwise disjoint paths $P_1,\dots,P_h$ starting in $A_1$ and ending in $B_r$.
	Each of these paths now has exactly one vertex in every $A_i$ and every $B_i$, $i \in [r]$.
	
	We now play the cut-matching game on the set $V = [h]$ with the strategy for the cut player of~\cref{thm:cut-matching}.
	The game lasts at most $r$ rounds.
	At each round $i \in [r]$, we simulate the matching player as follows: For a given partition $\Brace{V^1,V^2}$ of $[h]$, we define for both $c=1,2$ the set $A_i^c \subseteq A_i$ as those vertices $v \in A_i$ that lie on a path $P_j$ with $j \in V^c$.
	Since the path-of-sets system is strong, $A_i$ is well-linked in $G[S_i]$.
	We thus use the well-linkedness of $A_i$ in $G[S_i]$ to obtain a $A_i^1$-$A_i^2$-linkage $\mathcal{Q}_i$ contained in $G[S_i]$.
	Finally, the matching player answers with a matching $M_i$ that corresponds to which elements of $A_i^1$ were connected to which elements of $A_i^2$.
	In formulas, \[M_i = \{ \alpha\beta \mid \alpha \in V^1, \beta \in V^2\text{, and } \mathcal{Q}_i \text{ contains a path with endpoints in } P_\alpha \text{ and } P_\beta \}\text{.}\]
	
	Let $H$ be the graph at the end of the cut-matching game.
	Since $H$ is an $\Omega(1)$-expander of maximum degree $\Oh((\log h)^2)$, by~\cref{thm:cliques-in-expanders}, it contains a clique minor of size $t$ where $t = \Omega(\sqrt{h}/(\log h)^2)$.
	Denote the model of this clique as $(K_a)_{a \in [t]}$.
	That is, the sets $K_a$, $a \in [t]$, are pairwise disjoint subsets of $V(H)$, each $H[K_a]$ is connected, and for every $ab \in \binom{[t]}{2}$ there exists $j_{a,b} \in K_a$ and $j_{b,a} \in K_b$ with $j_{a,b}j_{b,a} \in E(H)$.
	Let $T_a$ be a spanning tree of $H[K_a]$.
	
	Since $H$ consists of the matchings $M_1,\ldots,M_r$, to every edge $e \in E(H)$ we can associate the path $Q(e) \in \mathcal{Q}_i$ that induced $e$, that is, $e \in M_i$, $e= \alpha\beta$, and $Q(e)$ has its endpoints on the paths $P_\alpha$ and $P_\beta$.
	
	For every $a \in [t]$, we construct a connected subgraph $G_a$ of $G$ consisting of
	the paths $P_\alpha$ for each $\alpha \in K_a$, 
	the paths $Q(e)$ for each $e \in E(T_a)$,
	and the paths $Q(j_{a,b}j_{b,a})$, excluding the endpoint on $P_{j_{b,a}}$, for all $a, b \in t$ with  $a < b \leq t$.
	Then, as $(K_a)_{a \in [t]}$ is a clique model in $H$,
	$(G_a)_{a \in [t]}$ is a bramble in $G$.
	Furthermore, $(G_a)_{a \in [t]}$ is of congestion $2$ as every vertex of $G$
	can lie on at most one path $P_\alpha$ and at most one path $Q(e)$. 
	Finally, $t = \Omega(\sqrt{h} / (\log h)^2) = \widetilde{\Omega}(\sqrt{k})$.
\end{proof}

\section{Upper bound for exponential brambles}\label{sec:lb}
In this section we prove \cref{thm:lb}. For convenience we restate it below.
\secondresult*

\paragraph{Treewidth sparsifier.}
It is straightforward to lift a bramble in a topological minor $H$ of a graph $G$
to a bramble in $G$ without decreasing the order of the bramble.
Thus, by hiding the polylogarithmic loss on the treewidth of \cref{thm:tw-sparsifier}
within the $p(\cdot)$ factor of \cref{thm:lb}, we can assume
that we consider a graph $G$ of treewidth at least $k$ and $|V(G)| = \widetilde{\Oh}(k^4)$.

\paragraph{Concurrent flow.}
Let us define a \emph{flow} $f_{u, v}$ between two vertices $u$ and $v$ as a weighted collection of pairwise distinct paths between $u$ and $v$ (if $u = v$ then the only possible path in $f_{u, v}$ has length 0).
The \emph{units of flow sent through a vertex $w$} by $f_{u, v}$ is the sum of the weights of the paths in $f_{u, v}$ that contain~$w$.
The \emph{value} of $f_{u, v}$ is defined as the units of flow sent through $u$ (or equivalently through $v$).
For a set $W \subseteq V(G)$, a \emph{concurrent flow of value $\nu$
	and congestion $\gamma$} is a collection of $|W|^2$ flows $(f_{u,v})_{(u,v) \in W \times W}$
such that
\begin{itemize}
	\item $f_{u,v}$ sends exactly $\nu$ units of flow from $u$ to $v$; and
	\item for each vertex $w$ the total flow over all flows $f_{u, v}$ sent through $w$ is at most $\gamma$.
\end{itemize}
We need the following well-known result on well-linked sets and multicommodity flows;
for a proof, see, e.g., the first paragraph of the proof of Lemma~14 by \cite{GroheM09}.
\begin{lemma}[\cite{GroheM09}]\label{lem:flow}
	Let $G$ be a graph of treewidth at least $k$. 
	Then there exists a set $W \subseteq V(G)$ of size at least $k/3$
	and a concurrent flow of value $1$ and congestion at most $\beta k \log k$, for some constant $\beta$ (not dependent on $k$).
\end{lemma}
\noindent Note that the sum of the values of all flows in a concurrent flow of value $1$ is $\Oh(k^2)$;
the essence of \cref{lem:flow} is that only a tiny part of it, an $\Oh((\log k )/ k)$ fraction, can pass through a single vertex.
Without loss of generality, we can assume that the constant $\beta$ of \cref{lem:flow} satisfies $\beta > 1/9$ (which we need for technical reasons later on).

\paragraph{Sampling a path.}
Let $G$ be a graph.
Similarly as~\cite{GroheM09}, we use the concurrent flow given by \cref{lem:flow}
to sample paths between vertices of a given subset of $V(G)$. 
From now on, let $W \subseteq V(G)$ be a set of size exactly $\lfloor \nicefrac{k}{3} \rfloor$ and $(f_{u,v})_{(u,v) \in W \times W}$ the concurrent flow
given by \cref{lem:flow} applied to $G$.
(Note that, if the size of the set $W$ promised by \cref{lem:flow} is larger than $\lfloor \nicefrac{k}{3} \rfloor$ we can omit vertices from $W$ and drop the corresponding flows from the concurrent flow obtaining a concurrent flow of the same value and at most the same congestion.)
For each $(u, v) \in W \times W$ let $\mathcal{P}_{u,v}$ be the family of paths in $f_{u, v}$ and let $f_{u,v}(P)$ be the flow value passed along a path $P \in \mathcal{P}_{u,v}$ (i.e.\ the weight of $P$ in $f_{u, v}$).
Since the value of $f_{u,v}$ is $1$, flow $f_{u,v}$ can be interpreted as a probability
distribution over $\mathcal{P}_{u,v}$.
\begin{claim}\label{cl:sample}
	Fix $x \in V(G)$.
	Assume that we are sampling two vertices $u,v \in W$ 
	uniformly at random
	and then sampling a path from $u$ to $v$ according to the following distribution:
	the probability of sampling $P \in \mathcal{P}_{u,v}$
	equals $f_{u,v}(P)$ (and paths not from $\mathcal{P}_{u,v}$ have zero probability). 
	Then, the probability that $x$ lies on a sampled path is at most $9 \beta \log k / k$.
\end{claim}
\begin{proof}
	The experiment in the statement is equivalent to sampling a pair $(u,v) \in W \times W$
	and then sampling a path $P \in \mathcal{P}_{u,v}$ according to $f_{u,v}$ as a probability
	distribution on $\mathcal{P}_{u,v}$.
	Since there are at most $k^2/9$ possible choices for $f_{u,v}$,
	the probability that $x$ is on the sampled path is bounded above by
	$9 \bar{f}(x) / k^2$, where $\bar{f}(x)$ is the total amount of flow passing through $x$. 
	Since $\bar{f}(x) \leq \beta k \log k$ by the congestion property of the concurrent flow $(f_{u, v})$, the claim follows.
	\cqed\end{proof}

\paragraph{Sampling a closed walk.}
Let $\ell := \lfloor \frac{k^{0.5+\delta}}{72\beta} \rfloor $.
By \emph{sampling a walk $\mathcal{W}$} we mean the following experiment. 
Sample uniformly and independently at random
$\ell$ vertices $s_1,s_2,\ldots,s_\ell \in W$
and then, for every $i \in [\ell]$, sample path $P_i \in \mathcal{P}_{s_i,s_{i+1}}$
according to $f_{s_i,s_{i+1}}$ (with indices cyclically modulo $\ell$). 
The walk $\mathcal{W}$ is then the concatenation of $P_1,P_2,\ldots,P_\ell$.
Note that the vertices $s_1, \ldots, s_\ell$ are not necessarily distinct.

We use \cref{cl:sample} in the following way.
\begin{claim}\label{cl:sample-cycle}
	Let $\alpha$ be a real number with $0 < \alpha < 1$ and let $X \subseteq V(G)$ be of size at most $\alpha \ell / \log k$.
	Sample a walk $\mathcal{W}$. Then,
	$$\mathrm{Pr}(X \cap V(\mathcal{W}) = \emptyset) \geq \exp\left(-\Oh(\alpha k^{2\delta})\right).$$
\end{claim}
\begin{proof}
	Let $s_1,\ldots,s_\ell$ and $P_1,\ldots,P_\ell$ be as in the definition of sampling a walk.
	
	Let $A_i$ be the event that $X \cap V(P_i) \neq \emptyset$.
	\Cref{cl:sample} with union bound over all $x \in X$ implies that
	\begin{equation}\label{eq:PrAi}
		\mathrm{Pr}(A_i) \leq \frac{9 \alpha \beta \ell}{k} \leq \frac{\alpha}{8 k^{0.5-\delta}}.
	\end{equation}
	Observe that $A_i$ is independent of $A_j$ unless $i=j+1$ or $i+1=j$ (with indices cyclically
	modulo~$\ell$). That is, $A_i$ is independent of all but two other events $A_i$. 
	
	Let $\pi(k)$ denote the right-hand side of~\cref{eq:PrAi}; note that $\pi(k) < 1/8$. 
	Then, $\mathrm{Pr}(A_i) \leq (4\pi(k)) \cdot (1-4\pi(k))^2$. 
	Hence, by the Lov\'{a}sz Local Lemma~(\cref{thm:lll}), we obtain that
	\begin{equation}\label{eq:PrCompAi}
		\mathrm{Pr}\left(\bigwedge_{i=1}^\ell \overline{A_i}\right) \geq \prod_{i=1}^\ell \left(1-4\pi(k)\right) \geq \left(1-\frac{\alpha}{2k^{0.5-\delta}}\right)^{\frac{k^{0.5+\delta}}{72 \beta}} = \left(\left(1-\frac{1}{\frac{2k^{0.5-\delta}}{\alpha}}\right)^{\frac{2k^{0.5 - \delta}}{\alpha}}\right)^{\frac{\alpha k^{2\delta}}{144\beta}}.
	\end{equation}
	From the folklore relation $(1 + 1/x)^x \leq e \leq (1 + 1/x)^{x + 1}$ for all positive $x$ we obtain that the right-hand side of~\cref{eq:PrCompAi} is in
	\begin{equation*}
		\exp\left(-\Oh\left(\frac{\alpha k^{2\delta}}{\beta}\right)\right) = \exp\left(-\Oh(\alpha k^{2\delta})\right).
	\end{equation*}
	This concludes the proof of the claim.
	\cqed\end{proof}

\paragraph{Many closed walks sampled independently form a bramble.}
\Cref{cl:sample-cycle} asserts that there is a nontrivial chance that a hitting set of
size at most $\alpha \ell / \log k$ misses a sampled walk. On the other hand, two walks sampled independently at random
intersect with high probability, so we can sample a large number of walks that pairwise intersect.
\begin{claim}\label{cl:many-walks}
	Let $\mathcal{W}_1$ and $\mathcal{W}_2$ be two walks sampled independently. 
	Then, 
	$$\mathrm{Pr}(V(\mathcal{W}_1) \cap V(\mathcal{W}_2) = \emptyset) \leq \exp\left(-\Omega(k^{2\delta})\right).$$
	Consequently, for some universal constant $\lambda > 0$, 
	if one samples a family $\mathcal{B}$ of $\lfloor \exp(\lambda k^{2\delta}) \rfloor$
	walks (each walk is sampled independently), then 
	the walks pairwise intersect with probability larger than $0.5$.
\end{claim}
\begin{proof}
	The second part of the claim follows directly from the first part by union bound.
	
	For the first part, let $s_1,\ldots,s_\ell$ be the vertices sampled in the process of sampling $\mathcal{W}_1$
	and $s_1',\ldots,s_\ell'$ be the vertices sampled in the process of sampling $\mathcal{W}_2$.
	It suffices to show that 
	$$\mathrm{Pr}(\{s_1,\ldots,s_\ell\} \cap \{s_1',\ldots,s_\ell'\} = \emptyset) \leq \exp\left(-\Omega(k^{2\delta})\right).$$
	
	Let $S = \{s_1,\ldots,s_\ell\}$.
	Observe that $\ell = \lfloor \frac{k^{0.5+\delta}}{72 \beta} \rfloor < \lfloor k/12 \rfloor \leq |W|/2$. 
	Hence,  for each $i \in [\ell]$ \[\mathrm{Pr}(s_i \in \{s_1,\ldots,s_{i-1}\}) < 0.5.\]
	Let $X_i = 1$ if $s_i \notin \{s_1,\ldots,s_{i-1}\}$ and $X_i = 0$ otherwise.
	Then, we have $|S| = \sum_{i=1}^\ell X_i$ and $\mathrm{Pr}(X_i = 1) > 0.5$.
	
	Let $X_1',X_2',\ldots,X_\ell'$ be independent symmetrical Bernoulli variables.
	From the previous paragraph we infer that for every real $r$ we have
	$\mathrm{Pr}(\sum_{i=1}^\ell X_i \geq r) \geq \mathrm{Pr}(\sum_{i=1}^\ell X_i' \geq r)$. 
	Hence, by the Chernoff inequality $\mathrm{Pr}(\sum_{i = 1}^\ell X_i' \leq (1 - a) \mu) \leq e^{-a^2\mu/2}$ for $0 \leq a \leq 1$, where $\mu = \ell/2$ is the expected value of $\sum_{i = 1}^\ell X_i'$, we have:
	$$\mathrm{Pr}(|S| < \ell / 4) = \mathrm{Pr}\left(\sum_{i=1}^\ell X_i < \ell/4\right) 
	\leq \mathrm{Pr}\left(\sum_{i=1}^\ell X_i' < \ell / 4\right) 
	\leq e^{-\ell/16}.$$
	The above allows us to condition now on the event $|S| \geq \ell/4$.
	Let us move to sampling vertices~$s_i'$.
	We have
	\[\mathrm{Pr}(\forall_{i=1}^\ell s_i' \notin S) = \left(1-\frac{|S|}{k}\right)^\ell 
	\leq \left(1-\frac{\ell/4}{k}\right)^\ell 
	\leq \exp\left(-\Omega(k^{2\delta})\right) \text{, as required.} \]
	\cqed
	\end{proof}
Let $\lambda > 0$ be the constant of \Cref{cl:many-walks}
and let $\mathcal{B}$ be a 
sequence of $\lfloor \exp(\lambda k^{2\delta}) \rfloor$ closed walks
sampled independently. 
\Cref{cl:many-walks} asserts that with probability more than $0.5$ the family
$\mathcal{B}$ (where every walk is interpreted as its vertex set) is a bramble in $G$.

\paragraph{Order of the sampled bramble.}
We now show that for a sufficiently small constant $\alpha > 0$,
with probability more than $0.5$
for every set $X \subseteq V(G)$ of size at most $\alpha \ell / \log k$ 
there exists $\mathcal{W} \in \mathcal{B}$ that is disjoint with $X$.
Together with the fact that $\mathcal{B}$ is a bramble with probability larger than $0.5$, 
this proves that with positive probability both $\mathcal{B}$ is a bramble and the minimum size of a hitting set of $\mathcal{B}$ is $\Omega(k^{0.5+\delta}/\log k)$,
concluding the proof of \cref{thm:lb}.

By \cref{cl:sample-cycle}, for a 
single walk $\mathcal{W} \in \mathcal{B}$, we have
$$\mathrm{Pr}(X \cap V(\mathcal{W}) = \emptyset) \geq \exp\left(-\Oh(\alpha k^{2\delta})\right).$$
Recall that the walks in $\mathcal{B}$ are sampled independently, so, for a fixed vertex subset $X \subseteq V(G)$,
the events $(X \cap V(\mathcal{W}) = \emptyset)_{\mathcal{W} \in \mathcal{B}}$ are independent. 
Hence, for a fixed vertex subset $X \subseteq V(G)$ of size at most $\alpha \ell / \log k$ we have
$$\mathrm{Pr}(\forall_{\mathcal{W} \in \mathcal{B} } X \cap V(\mathcal{W}) \neq \emptyset) 
\leq \left(1 - \exp\left(-\Oh(\alpha k^{2\delta})\right)\right)^{\lfloor \exp(\lambda k^{2\delta}) \rfloor}
\leq \exp\left(- \exp(\lambda/2 \cdot k^{2\delta})\right).$$
\newcommand{\W}{\ensuremath{\mathcal{W}}}



Here, in the last inequality we have chosen $\alpha$ small enough so that
$\alpha$ times the constant hidden in the big-$\Oh$ notation is smaller than $\lambda/2$.
As $|V(G)| = \widetilde{\Oh}(k^4)$, there are $2^{\Oh(\alpha k^{0.5+\delta})}$ choices
of a set $X \subseteq V(G)$ of size at most $\alpha \ell / \log k$.
By union bound over the possible sets $X$ we obtain:
$$\mathrm{Pr}\left( \exists_{X \subseteq V(G)} |X| \leq \alpha \ell / \log k \wedge 
\forall_{\mathcal{W} \in \mathcal{B}} X \cap V(\mathcal{W}) \neq \emptyset \right)
\leq \exp\left(\Oh(\alpha k^{0.5+\delta}) - \exp(\lambda/2 \cdot k^{2\delta})\right).$$
By choosing $\alpha$ sufficiently small, the last expression can be made to be less than $0.5$.

Hence, with probability more than $0.5$ the family $\mathcal{B}$ does not admit a hitting set of size at most $\alpha \ell/\log k$
and with probability more than $0.5$ the family $\mathcal{B}$ is a bramble.
By the union bound, $\mathcal{B}$ is a bramble of order $\widetilde{\Omega}(k^{0.5+\delta})$ with positive probability.
This concludes the proof of \cref{thm:lb}.

\nocite{*}
\bibliographystyle{abbrvnat}
\bibliography{literature}
\label{sec:biblio}

\end{document}